\documentclass{aims}
\usepackage{amsmath,amssymb}
  \usepackage{paralist}
  \usepackage{graphics} %% add this and next lines if pictures should be in esp format
  \usepackage{epsfig} %For pictures: screened artwork should be set up with an 85 or 100 line screen
\usepackage{graphicx}  \usepackage{epstopdf}%This is to transfer .eps figure to .pdf figure; please compile your paper using PDFLeTex or PDFTeXify.
 \usepackage[colorlinks=true]{hyperref}
\hypersetup{urlcolor=blue, citecolor=red}

\def\currentvolume{X}
 \def\currentissue{X}
  \def\currentyear{200X}
   \def\currentmonth{XX}
    \def\ppages{X--XX}
     \def\DOI{10.3934/xx.xx.xx.xx}

\newtheorem{theorem}{Theorem}[section]

\newtheorem{lemma}[theorem]{Lemma}
\newtheorem{proposition}{Proposition}

\theoremstyle{definition}
\newtheorem{definition}[theorem]{Definition}
\newtheorem{remark}{Remark}

\newcommand{\stft}{short-time Fourier transform}

\def\epsilon{\varepsilon}

\newcommand{\beqa}{\begin{eqnarray*}}
\newcommand{\eeqa}{\end{eqnarray*}}

\newcommand{\field}[1]{\mathbb{#1}}
\newcommand{\bR}{\field{R}}        %  real numbers
\def\la{\lambda}

\def\cS{\mathcal{S}}

\def\cM{\mathcal{M}}

\def\rd{\bR^d}

\def\rdd{{\bR^{2d}}}

\def\L{\left(}
\def\R{\right)}

\def\<{\left<}
\def\>{\right>}

\def\mv1{M_v^1}

\def\phas{(x,\o )}
\def\mn{(m,n)}
\def\mn'{(m',n')}

\renewcommand{\qed}{\hfill \rule{7pt}{8pt} \vskip .2truein}   % Q.E.D. box

\def\o{\eta}

\def\R{\mathbb{R}}
\def\Ren{\mathbb{R}^d}

\def\Fur{\mathcal{F}}

\def\Sn2{S_{2}(L^{2}(\Ren))}
\def\S1{S_{1}(L^{2}(\Ren))}
\def\sig00{\sigma_{0,0}}
\def\la{\langle}
\def\ra{\rangle}

\title[Schr\"odinger equations with rough Hamiltonians] %Use the shortened version of the full title
      {Schr\"odinger equations with rough Hamiltonians}

\author[{Elena Cordero,  Fabio Nicola  and Luigi Rodino}]{}


\email{elena.cordero@unito.it}
\email{fabio.nicola@polito.it}
\email{luigi.rodino@unito.it}

\begin{document}
\maketitle

% Enter the first author's name and address:
\centerline{\scshape Elena Cordero}
\medskip
{\footnotesize
% please put the address of the first author
 \centerline{Dipartimento di Matematica,
Universit\`a di Torino}
  % \centerline{Other lines}
   \centerline{Via Carlo Alberto 10, 10123 Torino, Italy}
} % Do not forget to end the {\footnotesize by the sign }

\medskip

\centerline{\scshape Fabio Nicola}
\medskip
{\footnotesize
 % please put the address of the second  and third author
 \centerline{Dipartimento di Scienze Matematiche,
Politecnico di Torino}
 %\centerline{Other lines}
   \centerline{Corso Duca degli Abruzzi 24, 10129 Torino, Italy}
}

\medskip

% Enter the first author's name and address:
\centerline{\scshape Luigi Rodino}
\medskip
{\footnotesize
% please put the address of the first author
 \centerline{Dipartimento di Matematica,
Universit\`a di Torino}
  % \centerline{Other lines}
   \centerline{Via Carlo Alberto 10, 10123 Torino, Italy}
} % Do not forget to end the {\footnotesize by the sign }

\bigskip

% The name of the associate editor will be entered by an editorial staff
% "Communicated by the associate editor name" is not needed for special issue.
 \centerline{(Communicated by the associate editor name)}

%The abstract of your paper
\begin{abstract}
We consider a class of linear Schr\"odinger equations in $\rd$ with rough Hamiltonian, namely with certain derivatives in the Sj\"ostrand class $M^{\infty,1}$. We prove that the corresponding propagator is bounded on modulation spaces. The present results improve  several contributions recently appeared in the literature and can be regarded as the evolution counterpart of the fundamental result of Sj\"ostrand about the boundedness of pseudodifferential operators with symbols in that class.\par
Finally we consider nonlinear perturbations of real-analytic type and we prove local wellposedness of the corresponding initial value problem in certain modulation spaces.\end{abstract}

%The title of your section 1
\section{Introduction}
It is well-known that the free Schr\"odinger propagator $e^{it\Delta}$ in $\R^d$ is not bounded on the Lebesgue spaces $L^p$, except for $p=2$. This has motivated the study in other function spaces arising in Harmonic Analysis. Among them recently much attention has been given to the so-called modulation spaces. They can be defined similarly to the Besov spaces, but dyadic annuli in the frequency domain are replaced by isometric boxes. Actually, for our purposes it will be more useful an equivalent definition, in terms of the short-time Fourier tranform, or Bargmann transform, of temperate distributions (see \cite{fei,book} or Section 2 below, where the weighted variant is considered). Namely, for $x,\omega\in\rd$, consider the translation and modulation operators  
\[
T_x f(y)=f(y-x),\quad M_\omega f(x)= e^{i\omega x} f(x),
\]
as well as the time-frequency shifts
\[
\pi(z) f= T_x M_\omega f, \quad z=(x,\omega)\in\rd\times\rd.
\]
The short-time Fourier tranform (STFT) of a temperate distribution $f\in\cS'(\rd)$ with respect to a Schwartz function $g\in\cS(\rd)$ is defined as
\[
V_gf(x,\omega)=\langle f, \pi(x,\omega) g\rangle
\]
(with the pairing $\langle \cdot,\cdot\rangle$ conjugate linear in the second factor). Then, for $1\leq p,q\leq\infty$ and $0\not\equiv g\in\cS(\rd)$, we define the modulation spaces
\[
M^{p,q}(\rd)=\{f\in\cS'(\rd): \|f\|_{M^{p,q}(\rd)}:=\|V_g f\|_{L^{p,q}}=\| V_g f\|_{L^{q}(\rd_\omega;L^p(\rd_x))}<\infty\},
\]
and $M^p(\rd)=M^{p,p}(\rd)$ (changing window yields equivalent norms). In particular we have $M^{2}(\rd)=L^2(\rd)$, whereas $L^2$-based Sobolev spaces can be regarded as weighted modulation spaces as well. In short, the modulation space norm measures the position-momentum (or time-frequency) concentration in phase space of a function.

%Namely, for $1\leq p,q\leq\infty$, $s\geq 0$, one defines \begin{equation}\label{prima} M^{p,q}(\rd)=\Big\{f\in\cS'(\rd): \|f\|_{M^{p,q}}:=\Big(\sum_{k\in\bZ^d}\|\square_k f\|_{L^p}^q\Big)^{1/q}<\infty\Big\} \end{equation} (with obvious changes if $q=\infty$), where $\square_k$ are Fourier multipliers with symbols $\chi_{\mathcal{Q}_k}$ conveniently smoothed (we set $M^{p}$ for $M^{p,p}$). 

Now, it was proved in \cite{bertinoro3,bertinoro57} that the propagator $e^{it\Delta}$ is in fact bounded on $M^{p,q}(\rd)$, $1\leq p,q\leq\infty$; see also \cite{kki1,kki2,kki3,kki4}. As shown in \cite{MNRTT}, in the case of the fractional Laplacian $(-\Delta)^{\kappa/2}$, with $\kappa>2$, a loss of derivatives instead occurs.\par
Local well-posedness  of the corresponding nonlinear equations, with nonlinearity of power-type, or even entire real-analytic, were also considered in \cite{bertinoro2,bertinoro12}. Remarkably, modulation spaces revealed to provide a good framework for the global wellposedness as well, as shown recently in \cite{baoxiang0,bertinoro57,bertinoro58,bertinoro58bis} for several dispersive equations. The main results in this connection are now available in the recent book \cite{baoxiang} or in the survey \cite{ruz}. \par
A strictly related issue is the sparsity of the Gabor matrix representation, in phase space, for the corresponding propagator; this property can be in fact considered as a microlocal form of boundedness and implies the boundedness on modulation spaces in the usual sense. However it contains much more refined information which is essential, e.g., for the problem of propagation of singularities; we refer to \cite{CNG,cgnr,cn2,cnr1,cnr2,g-ibero,GR,nicola,rochberg,tataru} and the references therein for more detail. \par
A case of special interest is given by the Schr\"odinger equation with potential, i.e.\ $D_t-\Delta+V(t,x)$, with $V$ real-valued. The literature about wellposedness in $L^2$-based Sobolev spaces is enormous and we refer e.g.\ to \cite{dancona} and the references therein. Concerning the wellposedness in modulation spaces, it was proved in \cite{kki4} that if $V$ is smooth and has quadratic growth, i.e.\ $\partial^{\alpha}_x V(t,x)\in L^\infty([0,T]\times\rd)$, for $|\alpha|\geq 2$, the corresponding propagator is bounded on $M^{p}(\rd)$, $1\leq p\leq\infty$ (this result also follows from \cite{cgnr}). Instead, if $V$ is ``first order'', i.e.\ $\partial^{\alpha}_x V(t,x)\in L^\infty([0,T]\times\rd)$, for $|\alpha|\geq 1$, one has boundedness on every $M^{p,q}(\rd)$, $1\leq p,q\leq\infty$. The main motivation of the present paper is an extension of these results in the case of non-smooth potential, say with minimal regularity. Indeed, our result will apply to more general equations. \par  Looking for optimal results, one is led to consider potentials with derivatives in the so-called Sj\"ostrand class, which is nothing but the modulation space $M^{\infty,1}(\rd)$. In fact, in the theory of pseudodifferential operator that function space was first introduced in \cite{sjostrand,wiener31} as a natural symbol class with minimal regularity which still gives rise to bounded operators on $L^2(\rd)$, and also on $M^{p,q}(\rd)$ \cite{book} (similar results hold for certain Fourier integral operators as well \cite{wiener6,wiener7,wiener8,cnr1}). Several deep results in Micolocal Analysis, such as the Fefferman-Phong inequality, keep valid for symbols (with a certain number of derivatives) in $M^{\infty,1}(\rdd)$ \cite{herau,lerner,lerner2}. Hence that class looks the natural choice when dealing with boundedness of linear operators in modulation spaces. \par
Let us now state our results in a simplified form for a model equation; we refer to Section 2 below for general statements, which involve weighted modulation spaces as well. \par
Denote by $\cM^{p,q}(\rd)$ and $\cM^p(\rd)$ the closure of the Schwartz space in the corresponding modulation spaces (hence $\cM^{p,q}(\rd)=M^{p,q}(\rd)$ if $p,q<\infty$). \par\begin{theorem}
Let $T>0$ be fixed. Consider the initial value problem  
\begin{equation}\label{equazione13bis}
\begin{cases}
D_tu+(-\Delta)^{\kappa/2}u+V_2(t,x) u+V_1(t,x) u+V_0(t,x) u=0\\
u(0)=u_0,
\end{cases}
\end{equation}
with $0<\kappa\leq 2$, $t\in[0,T]$, $x\in\rd$ ($D_t=-i\partial_t$). Suppose that $\partial^\alpha_x V_2(t,\cdot)\in M^{\infty,1}(\rd)$ for $|\alpha|=2$, $\partial^\alpha_x V_1(t,\cdot)\in M^{\infty,1}$ for $|\alpha|=1$, $V_0\in M^{\infty,1}(\rd)$, with $V_2$ and $V_1$ real-valued. Assume moreover a mild continuity dependence with respect to $t$ (narrow convergence; see Section 2 below) and let $1\leq p\leq \infty$.\par Then for every $u_0\in \mathcal{M}^p(\rd)$, there exists a unique solution $u\in C([0,T],\mathcal{M}^p(\rd))$ to \eqref{equazione13bis}. Moreover the corresponding propagator is bounded on $\mathcal{M}^p(\rd)$.
\end{theorem}
In the case of first order potentials we have a better conclusion.
\begin{theorem} Suppose in addition $V_2\equiv0$ and let $1\leq p,q\leq \infty$. Then for every $u_0\in \mathcal{M}^{p,q}(\rd)$ there exists a unique solution $u\in C([0,T],\mathcal{M}^{p,q}(\rd))$ to \eqref{equazione13bis}. Moreover the corresponding propagator is bounded on $\mathcal{M}^{p,q}(\rd)$.
\end{theorem}
It is shown in Remark \ref{remark5.2} that similar results do not hold if one replaces $\cM^\infty(\rd)$ by the larger space $M^\infty(\rd)$. 

%Partial results for non-smooth potentials have already appeared in \cite{wei}. There the author proves that the above conclusions hold assuming $\partial^{\alpha}_x V_2(t,x)\in L^\infty$ for $2\leq |\alpha|\leq 2[d/2]+4$, $V_1\equiv V_0\equiv 0$, for boundedness on $M^{p}(\rd)$, or $\partial^{\alpha}_x V_1(t,x)\in L^\infty$ for $1\leq |\alpha|\leq 2[d/2]+4$, $V_2\equiv V_0\equiv 0$,  for boundedness on $M^{p,q}(\rd)$. 

For comparison, observe that $C^{d+1}(\rd)\subset M^{\infty,1}(\rd)\subset L^\infty(\rd)$, but functions in $M^{\infty,1}$ generally do not possess any derivative. Hence, the above results represent a significant improvement of those in \cite{kki4}. \par
Notice that in the special case of $0$th order potentials $(V_2=V_1\equiv0)$ more refined results were obtained in \cite{cgnr2}, where the propagator was shown to be a generalized metaplectic operator; see also \cite{cnr3} for the corresponding problem of propagation of singularities.\par   
Actually, as anticipated, we will consider much more general equations, namely of the form $D_t u+ a^w(t,x,D_x)u=0$,
where $a(t,x,\xi)$ is a second order pseudodifferential operator, whose symbol has some derivatives in $M^{\infty,1}(\rdd)$. In this connection our results can be regarded as the evolution counterpart of the boundedness results for pseudodifferential operators with symbols in $M^{\infty,1}(\rdd)$ proved in \cite{sjostrand,book}. Moreover, to our knowledge, they contain as special cases all the known continuity results of Schr\"odinger propagators on modulation spaces.\par
 The proof of wellposedness relies on the construction of a parametrix for the forward Cauchy problem, in the spirit of \cite{kt,smith,staffilani,tataru,tataru2}. Namely, one decomposes the initial datum in coherent states $T_xM_\omega g$, i.e.\ Gabor atoms, where $g$ is a fixed window; then a parametrix is constructed as a generalized localization operator in phase space (a type of operators introduced in their basic form in \cite{daubechies}) that moves the Gabor atoms in phase space according to the Hamiltonian flow, together with a phase shift. Notice that at this low level of regularity the more classical approach via WKB expansions and Fourier integral operators turns out inapplicable.
\par Finally, we will consider the corresponding nonlinear equations, with a nonlinearity $F(u)$ which is entire real-analytic in $\mathbb{C}$ (e.g.\ a polynomial in $u,\overline{u}$) and we will prove that all the above results extend in that setting if the initial datum is in $\cM^1(\rd)$ or even in $\cM^{p,1}(\rd)$ (in the case of first order potentials), at least for small time.\par\medskip
Briefly, the paper is organized as follows. In Section 2 we recall the main definition and properties of modulation spaces and we state the results in full generality. Section 3 is devoted to some preliminary estimates, whereas in Section 4 we introduce a class of generalized localization operators which will appear subsequently. Section 5 is devoted to the proof of the main results (linear case). Finally Section 6 deals with the extension to nonlinear equations, at least for small time.

\section{Notation and statement of the results}
\subsection{Weyl quantization \cite{hormanderIII}} 
The Fourier transform is normalized as
\[
\widehat{f}(\xi)=\Fur f(\xi)=\int_{\rd} e^{-ix\xi} f(x)\, dx,
\]
and the Weyl quantization of a symbol $a(x,\xi)$ is correspondingly defined as
\[
a^w f(x)=a^w(x,D) f=\iint_{\rdd} e^{i(x-y)\xi} a\Big(\frac{x+y}{2},\xi\Big) f(y) dy\, d\xi.
\]
We recall the following easy properties, which can be checked directly:
\begin{equation}\label{comm}
[a^w,x_j]=(D_{\xi_j} a)^w,\qquad [a^w,D_j]=-(D_{x_j}a)^w 
\end{equation}
\begin{equation}\label{comm1}
(x_j a)^w f=a^w (x_j f)-\frac{1}{2}(D_{\xi_j} a)^w f
\end{equation}
\begin{equation}\label{comm2}
(\xi_j a)^w f=\frac{1}{2} (D_{x_j}a)^w f+a^w(D_{x_j} f).
\end{equation}
\subsection{Modulation spaces \cite{fei,book}} We have already defined in the Introduction the time-frequency shifts and the STFT of a temperate distribution, as well as the unweighted modulation spaces $M^{p,q}(\rd)$. Here we extend the definition in the presence of a weight.\par 
We consider a positive submultiplicative even continuous function $v$ in $\rdd$ ($v(z+w)\lesssim v(z)v(w)$, $v(-z)=v(z)$); moreover we suppose in the sequel that $v$ satisfies 
\begin{equation}\label{condpesi}
v(\tau z)\leq  C v(z),\quad 0<\tau<1,\ z\in\rdd
\end{equation}
as well as 
\begin{equation}\label{condpesi2}
v(z)\leq C\langle z\rangle ^N
\end{equation}
for some $C,N>0$.\par
We denote by $\cM_v$ the space of $v$-moderate weight $m$, i.e.\ positive continuous  function $m>0$ in $\rdd$ such that $m(z+w)\lesssim v(z)m(w)$, satisfying in addition the following condition: for every constant $C_1>0$ there exists $C_2>0$ such that, for $z,w\in\rdd$,
\begin{equation}\label{condpesi3}
|z|\leq C_1 |w| \Longrightarrow m(z)\leq C_2 m(w).
\end{equation}
This implies that $m\circ\chi\asymp m$ if $\chi:\rdd\to\rdd$ is any invertible transformation, Lipschitz together with its inverse.  As prototype one can consider the standard weights
\[
v_r(z)=\langle z\rangle^r=(1+|z|^2)^{r/2},\quad z\in\rdd,
\]
for which we have $v_s\in \cM_{v_r}$ if and only if $|s|\leq r$. \par
Now, for $1\leq p,q\leq \infty$, $m\in\cM_v$ and $0\not\equiv g\in\cS(\rd)$, we define the spaces
\[
M^{p,q}_m(\rd)=\{f\in\cS'(\rd): \|f\|_{M^{p,q}_m(\rd)}:=\|V_g f\|_{L^{p,q}_m}=\| \|mV_g f\|_{L^{q}(\rd_\omega;L^p(\rd_x))}<\infty\}.
\]
It can be proved that changing windows in $\cS(\rd)$ (on even in $M^1_v(\rd)$) yields equivalent norms.
Briefly we write $M^p_m(\rd)=M^{p,p}_m(\rd)$.  Also, we set $\cM^{p,q}_m(\rd)$ for the closure of the Schwartz space in $M^{p,q}_m(\rd)$ (so that $\cM^{p,q}_m(\rd)=M^{p,q}_m(\rd)$ when $p,q<\infty$), and $\cM^p_m(\rd)=\cM^{p,p}_m(\rd)$. \par
We now recall the definition of the narrow convergence. In this connection there are several related definitions in the literature; the present one has first appeared in \cite{sjostrand} and is different e.g.\ from that in \cite{wiener8}.  
\begin{definition} Let $\Omega$ be a subset of some Euclidean space. We say that a map $\Omega\ni \zeta\mapsto a_\zeta\in M^{\infty,1}_{1\otimes v}(\rdd)$ is continuous for the narrow convergence if it is continuous in $\cS'(\rdd)$ (weakly) and if there exists a function $H\in L^1_v(\rdd)$ such that $\sup_{z\in \rdd}|V_g a_\zeta (z,w)|\leq H(w)$ for every $\zeta\in\Omega$ and almost every $w\in\rdd$. 
\end{definition}
Here $0\not\equiv g\in\cS(\rdd)$ is a fixed window and the definition is independent of the choice of $g$, as a consequence of Lemma \ref{changewind} below. It is also clear from the very definition that the set of symbols $\{a_\zeta:\,\zeta\in\Omega\}$ is then bounded in $ M^{\infty,1}_{1\otimes v}(\rdd)$. 
\subsection{Statement of the results}
 Let $T>0$ be fixed. Consider the Cauchy problem 
 \begin{equation}\label{equazione}
\begin{cases}
D_t u+ a^w(t,x,D_x)u=0\\
u(0)=u_0,
\end{cases}
\end{equation}
with $t\in [0,T]$, $x,\xi\in\rd$, and $a$ specified below.\par This is our main result. 
 \begin{theorem}\label{teo1} Let $a(t,x,\xi)=\sum_{j=0}^2 a_j(t,x,\xi)$, $t\in [0,T]$, $x,\xi\in\rdd$, where $a_2$, $a_1$ are real-valued and suppose that, for $j=0,1,2$ and  $j\leq|\alpha|\leq 2j$, the map 
\begin{equation}\label{due}
t\mapsto \partial^\alpha a_j(t,\cdot)\textit{ is continuous in }M^{\infty,1}_{1\otimes v}(\rdd)\textit{ for the narrow convergence}.
\end{equation}
Let $1\leq p\leq \infty$, $m\in \cM_v$. For every $u_0\in \mathcal{M}^p_m(\rd)$, there exists a unique solution $u\in C([0,T],\mathcal{M}^p_m(\rd))$ to \eqref{equazione}. Moreover the corresponding propagator is bounded on $\mathcal{M}^p_m(\rd)$.
\end{theorem}
In the case of symbols that can be written as a sum of symbols depending only on $x$ or $\xi$ less regularity may be assumed. 

 \begin{theorem}\label{teo2} Let $a(t,x,\xi)=\sum_{j=0}^2 \sigma_j(t,\xi)+V_j(t,x)$, $t\in [0,T]$, $x,\xi\in\rdd$, where $\sigma_j$, $V_j$, $j=1,2$, are real-valued and suppose that, for $j=0,1,2$ and  $|\alpha|=j$, the map
\begin{equation}\label{tre}
t\mapsto \partial^\alpha \sigma_j(t,\cdot), \partial^\alpha V_j(t,\cdot)\textit{ is continuous in }M^{\infty,1}_{1\otimes v}(\rd)\textit{ for the narrow convergence}.
\end{equation}
Let $1\leq p\leq \infty$, $m\in \cM_v$. For every $u_0\in \mathcal{M}^p_m(\rd)$, there exists a unique solution $u\in C([0,T],\mathcal{M}^p_m(\rd))$ to \eqref{equazione}. Moreover the corresponding propagator is bounded on $\mathcal{M}^p_m(\rd)$.
\end{theorem}
\begin{theorem}\label{teo3}
Under the same assumption as in Theorem \ref{teo2}, suppose moreover that $V_2\equiv0$. Let $1\leq p,q\leq \infty$, $m\in\cM_v$. For every $u_0\in \mathcal{M}^{p,q}_m(\rd)$, there exists a unique solution $u\in C([0,T],\mathcal{M}^{p,q}_m(\rd))$ to \eqref{equazione}. Moreover the corresponding propagator is bounded on $\mathcal{M}^{p,q}_m(\rd)$.
\end{theorem}
Theorem \ref{teo2} applies of course to problems of the form 
\begin{equation}\label{equazione2}
\begin{cases} 
D_tu-\Delta u+V_2(t,x) u+V_1(t,x)u+V_0(t,x)u=0\\
u(0)=u_0,
\end{cases}
\end{equation}
under the hypotheses given there on the potentials $V_j$, $j=0,1,2$. If $V_2\equiv0$ then Theorem \ref{teo3} applies as well. We can also consider the case of the fractional Laplacian, i.e.\ 
\begin{equation}\label{equazione3}
\begin{cases} 
D_tu+(-\Delta)^{\kappa/2} u+V_2(t,x) u+V_1(t,x)u+V_0(t,x)u=0\\
u(0)=u_0,
\end{cases}
\end{equation}
with $0<\kappa\leq 2$. 
\begin{theorem}\label{teo4} Consider potentials $V_j(t,x)$, $j=1,2,3$ with $V_1$, $V_2$ real-valued, and suppose that, for $j=0,1,2$, $|\alpha|=j$, $0\leq r<\kappa$, the map
 \[
 t\mapsto \partial^\alpha_x V_j(t,\cdot)\textit{ is continuous in }M^{\infty,1}_{1\otimes v_r}(\rd)\textit{ for the narrow convergence}.
 \]
  Let $1\leq p\leq \infty$, $m\in \cM_{v_r}$. For every $u_0\in \mathcal{M}^p_m(\rd)$, there exists a unique solution $u\in C([0,T],\mathcal{M}^p_m(\rd))$ to \eqref{equazione3}, the corresponding propagator being bounded on $\mathcal{M}^p_m(\rd)$.\par
Suppose moreover $V_2\equiv0$. Let $1\leq p,q\leq \infty$, $m\in\cM_{v_r}$. For every $u_0\in \mathcal{M}^{p,q}_m(\rd)$, there exists a unique solution $u\in C([0,T],\mathcal{M}^{p,q}_m(\rd))$ to \eqref{equazione3}, the corresponding propagator being bounded on $\mathcal{M}^{p,q}_m(\rd)$.
\end{theorem}
\section{Preliminary estimates}
In the sequel we will use the following covariance property of the STFT, which can be verified by direct inspection:
\begin{equation}\label{covarianza}
|\langle \pi(z) f,\pi(w)g\rangle|=|V_g f(w-z)|,\quad z,w\in\rdd.
\end{equation}
We also recall the following pointwise inequality of the \stft\
\cite[Lemma 11.3.3]{book}. It is 
useful when one needs to change window functions.
\begin{lemma}\label{changewind}
If  $g_0,g_1,\gamma\in\cS(\rd)$ such
that $\la \gamma, g_1\ra\not=0$ and
$f\in\cS'(\rd)$,  then the inequality
$$|V_{g_0} f\phas|\leq\frac1{|\la\gamma,g_1\ra|}(|V_{g_1} f|\ast|V_{\gamma} g_0|)\phas
$$
holds pointwise for all $\phas\in\rdd$.
\end{lemma}
The following lemma is proved in \cite[Lemma 3.1]{g-ibero}.
\begin{lemma}\label{lemmag}
Let $g\in\cS(\rd)$, $a\in\cS'(\rdd)$. For some window $\Phi\in\cS(\rdd)$ it turns out, setting $j(z_1,z_2)=(z_2,-z_1)$, 
 \[
 |\langle a^w(x,D) \pi(z)g,\pi(w)g \rangle|\leq \sup_{\omega\in\rdd} |V_{\Phi}a(\omega,j(w-z))|,\qquad z,w\in\rdd.
 \]
 \begin{proposition}\label{proregolarita} Let $[0,T]\ni t\mapsto a_t$ be a continuous map in $M^{\infty,1}_{1\otimes v}(\rdd)$ for the narrow convergence. Then $a_t(x,\xi)$ is continuous as a function of $t, x,\xi$. \end{proposition}
  \begin{proof} Clearly it suffices to consider the un-weighted case ($v\equiv1$), since $M^{\infty,1}_{1\otimes v}\hookrightarrow M^{\infty,1}$.  Let $z=(x,\xi)$, $z_0=(x_0,\xi_0)$. Then 
  \begin{equation}\label{30-u}
|\langle a_t,\delta_z\rangle-\langle a_{t_0},\delta_{z_0}\rangle|\leq |\langle a_t-a_{t_0},\delta_{z_0}\rangle| +|\langle a_t,\delta_z-\delta_{z_0}\rangle|.
\end{equation}
Consider now the first term in the right-hand side of \eqref{30-u}. If $g\in\cS(\rd)$, $\|g\|_{L^2}=1$, and $g^\ast(x):=\overline{g(-x)}$, it turns out 
\begin{align*}
|\langle a_t-a_{t_0},\delta_{z_0}\rangle|&=|\langle V_g(a_t-a_{t_0}),V_g(\delta_{z_0})\rangle|\\
&\leq \int |V_g (a_t-a_{t_0})(u,v)||T_{z_0}{g}^{\ast}(u)|\, du\, dv.
\end{align*}
This last expression tends to zero as $t\to t_0$ by the dominated convergence theorem, because $a_t\to a_{t_0}$ in $\cS'(\rdd)$ implies that $V_g (a_t-a_{t_0})\to0$ pointwise, whereas the assumption about narrow convergence yields 
\[
 |V_g (a_t-a_{t_0})(u,v)|\leq H(v),\quad t\in[0,T],\ u,v\in\rdd
\]
for some function $H\in L^1(\rdd)$. \par
Similarly, for the second term in the right-hand side of \eqref{30-u} we have 
\begin{align*}
|\langle a_t,\delta_z-\delta_{z_0}\rangle|&=|\langle V_g(a_t),V_g(\delta_{z}-\delta_{z_0})\rangle|\\
&\leq \int H(v)|e^{-2\pi i z v}T_{z}{g}^{\ast}(u)-e^{-2\pi i z_0 v}T_{z_0}{g}^{\ast}(u)|\, du\, dv
\end{align*}
for some $H\in L^1(\rdd)$, and one concludes as above.
\end{proof}

\rm The following result is essentially known. We shall give the proof for the benefit of the reader. 
 \begin{proposition}\label{proiniziale}
 Let $m\in\cM_v$, $1\leq p,q\leq\infty$. The operators $x_j, D_j$, $j=1,\ldots,d$, are bounded $\mathcal{M}^{p,q}_{v_1 m}(\rd)\to \mathcal{M}^{p,q}_{m}(\rd)$, $\mathcal{M}^{p,q}_{ m}(\rd)\to \mathcal{M}^{p,q}_{m/v_1}(\rd)$. Similarly, the operators $x^\alpha D^\beta$, $|\alpha|+|\beta|\leq 2$, are bounded $ \mathcal{M}^{p,q}_{v_2 m}(\rd)\to \mathcal{M}^{p,q}_m(\rd)$, $\mathcal{M}^{p,q}_m(\rd)\to \mathcal{M}^{p,q}_{m/v_2}(\rd)$
 \end{proposition}
 \begin{proof}
 It suffices to prove the first part of the statement. Consider the case of $x_j$ (similar arguments apply to $D_j$). Let $g\in\cS(\rd)$, $\|g\|_{L^2}=1$. We have 
 \begin{equation}\label{eqiniziale}
V_g(x_j f)(w)=\int\langle x_j \pi(z) g,\pi(w) g\rangle V_g f(z)\, dz.
\end{equation}
On the other hand, with $G(x)=x_j g(x)$,
\begin{align}\label{a19}
|\langle x_j \pi(z) g,\pi(w) g\rangle|&\leq |z||\langle\pi(z)g,\pi(w)g\rangle|+|\langle \pi(z)G,\pi(w) g\rangle|\nonumber\\
&\lesssim   |z|\langle z-w\rangle^{-M}.
\end{align}
for every $M>0$, where we used \eqref{covarianza} and the fact that the STFT of Schwartz functions is Schwartz.\par
Hence we get 
\[
|V_g(x_j f)|\lesssim v_{-M} \ast \Big(|V_g f|v_1\Big),
\]
and therefore 
\[
\|V_g (x_jf)\|_{L^{p,q}_m}\lesssim \|v_{-M}\|_{L^1_{v}}\|V_g f\|_{L^{p,q}_{mv_1}},
\]
which gives the desired boundedness of $x_j:\cM^{p,q}_{mv_1}(\rd)\to \cM^{p,q}_m(\rd)$ ($\|v_{-M}\|_{L^1_{v}}<\infty$ if $M$ is large enough, by \eqref{condpesi2}).
Similarly, from \eqref{a19} we have
\[
|\langle x_j \pi(z) g,\pi(w) g\rangle|\lesssim  |w|\langle z-w\rangle^{-M},
\]  
which gives the boundedness of $x_j:\cM^{p,q}_{m}(\rd)\to \cM^{p,q}_{m/v_1}(\rd)$
 \end{proof}
 
{\rm The following result shows the usefulness of the notion of narrow convergence.}  
\begin{proposition}\label{pro3.2} Let $\Omega\ni \zeta\mapsto a_\zeta\in M^{\infty,1}_{1\otimes v}(\rdd)$ be a continuous map from the narrow convergence. Let $m\in\cM_v$. Then the map of operators $\zeta\mapsto a_\zeta^w(x,D)$ is strongly continuous on $\mathcal{M}^{p,q}_m(\rd)$, $1\leq p,q\leq\infty$. 
\end{proposition}
\begin{proof}
It is well known, see e.g.\ \cite{g-ibero}, that the set $\{a_\zeta^w(x,D)\}$ is bounded in the space of linear continuous operators on $\mathcal{M}^{p,q}_m$. Hence it is sufficient to prove the strong continuity on $\mathcal{M}^1_m$, since $\mathcal{M}^1_m\subset \mathcal{M}^{p,q}_m$ with inclusion continuous and dense.\par
Now, we can suppose that $a_\zeta\to 0$, say for $\zeta\to\zeta_0$, for the narrow convergence. Let $g\in\cS(\rdd)$, $\|g\|_{L^2}=1$. For $f\in \mathcal{M}^1_m$, we get
\[
V_g(a_\zeta^w(x,D) f)(w)=\int \langle a_\zeta ^w(x,D) \pi(z)g,\pi(w)g \rangle V_g f(z)\, dz. 
\]
Hence
\begin{multline}\label{mul1}
\|a_\zeta^w(x,D) f\|_{\mathcal{M}^1_m}= \|V_g(a_\zeta^w(x,D) f)m\|_{L^1}\\ 
\leq \iint |\langle a_\zeta ^w(x,D) \pi(z)g,\pi(w)g \rangle|v(z-w) |V_g f(z)|m(z)  dz\, dw.
\end{multline}
Now, we have $\langle a_\zeta ^w(x,D) \pi(z)g,\pi(w)g \rangle\to0$ pointwise if $\zeta\to\zeta_0$. On the other hand it follows from Lemma \ref{lemmag} that for a new window $\Phi\in\cS(\rdd)$, with $j(z_1,z_2)=(z_2,-z_1)$, 
 \[
 |\langle a_\zeta^w(x,D) \pi(z)g,\pi(w)g \rangle|\leq \sup_{\omega\in\rdd} |V_{\Phi}(a_\zeta)(\omega,j(w-z))|\leq H(w-z),
 \]
 for some $H\in L^1_v(\rdd)$, where for the last inequality we used the hypothesis of narrow continuity of $\zeta\mapsto a_\zeta$. Since we also have $m V_g f\in L^1$ by assumption, we get $\|a_\zeta^w(x,D) f\|_{\mathcal{M}^1_m}\to 0$ from \eqref{mul1} and the dominated convergence theorem.
\end{proof}

The next results will be used often in the subsequent sections. 
\begin{proposition}\label{prod}
Let $[0,T]\ni t\mapsto a_t\in M^{\infty,1}_{1\otimes v}(\rdd)$ be continuous for the narrow convergence, and $\gamma(\tau)$ be a continuous function of $\tau\in[0,1]$.  Then the map 
\[
[0,T]\times \rdd\ni (t,\zeta)\mapsto b_{t,\zeta}(w):=\int_0^1 \gamma(\tau) a_t(\zeta+\tau w)\, d\tau,\quad w\in\rdd
\]
is still continuous in $M^{\infty,1}_{1\otimes v}(\rdd)$ for the narrow convergence. \par
The same holds true for the map
\[
[0,T]\times \rdd\ni (t,\zeta)\mapsto \tilde{b}_{t,\zeta}(w):=a_t(\zeta+w).
\]
\end{proposition}
\begin{proof}
We prove only the first part of the statement, because the last part follows similarly from an easier argument. \par
The continuity of $(t,\zeta)\mapsto b_{t,\zeta}$ in $\cS'(\rdd)$ is clear. Let us estimate the STFT of $b_{t,\zeta}$. Let $g$ be a Gaussian function, with $\|g\|_{L^2}=1$. We have 
\begin{align*}
|V_{g\otimes g}(b_{t,\zeta})(z,w)|&=\Big|\int _0^1 \gamma(\tau)V_{g\otimes g}[ a_t(\zeta+\tau\cdot)](z,w)\, d\tau\Big|\\
&\leq\int _0^1 |\gamma(\tau)|\tau^{-2d}|V_{g_{\tau}\otimes g_{\tau}}[a_t](\zeta+\tau z,\tau^{-1}w)|\, d\tau,
\end{align*}
where $g_\tau(x)=g(x/\tau)$.\par Using the change-of-window formula in Lemma \ref{changewind} we get
\[
|V_{g\otimes g}(b_{t,\zeta})(z,w)|\lesssim \int _0^1 |\gamma(\tau)|\tau^{-2d}\Big(|V_{g\otimes g}[a_t]|\ast |V_{g} g_\tau\otimes V_g g_\tau|\Big)(\zeta+\tau z,\tau^{-1}w)|\, d\tau.
\] 
We now take the supremum with respect to $z$ and we get, by Young inequality,  
\begin{equation*}
\sup_{z\in\rdd}|V_{g\otimes g}(b_{t,\zeta})(z,w)|\lesssim \int_0^1 |\gamma(\tau)| \tau^{-2d}  \Big(\sup_{z\in\rdd} | V_{g\otimes g}[a_t](z,\cdot)|\ast |V_g g_\tau|\Big)(\tau^{-1}w)\, d\tau,
\end{equation*}
where we used the estimate (see e.g.\ \cite[Formula (19)]{cn})
\begin{equation}\label{mul4}
\|V_g g_\tau\|_{L^1}\leq C,\quad 0< \tau\leq 1.
\end{equation}
Now, we have $ \sup_{z\in\rdd} | V_{g\otimes g}[a_t](z,w)|\leq H(w)$ for some $H\in L^1_v(\rdd)$, by assumption. Hence we obtain
\begin{equation}\label{mul3}
\sup_{z\in\rdd}|V_{g\otimes g}(b_{t,\zeta})(z,\cdot)|\lesssim \int_0^1 |\gamma(\tau)| \tau^{-2d} \Big(H\ast |V_g g_\tau|\Big)(\tau^{-1}\cdot)\, d\tau.
\end{equation}
It suffices to prove that this last expression is in $L^1_v$. This follows at once from the Fubini-Tonelli theorem, using the following two estimates: it turns out
\[
v(\tau\cdot)\big(H\ast|V_g g_{\tau}|\big)\lesssim (v(\tau\cdot)H)\ast|v(\tau\cdot)V_g g_{\tau}|\lesssim (vH)\ast|v(\tau\cdot)V_g g_{\tau}|,
\]
which follows because $v$ is submultiplicative and satisfies \eqref{condpesi}, and we also have
\[
\|v(\tau\cdot) V_g g_{\tau}\|_{L^1(\rdd)}\leq C,\quad 0< \tau\leq 1.
\]
In fact, an explicit computation (see e.g.\ \cite[Lemma 3.1]{cn}) shows that $| V_g g_{\tau}(z)|\lesssim \tau^{2d}\phi(\tau z)$ with $\phi$ Gaussian, whereas $v$ has polynomial growth.  
\end{proof}
\begin{proposition}\label{pro3.4} Let $a(t,\cdot)$ satisfy the assumption in Theorem \ref{teo1}, and $m\in\mathcal{M}_v$, $1\leq p,q\leq\infty$. Then the operator $a^w(t,x,D_x)$ is bounded $\mathcal{M}^{p,q}_{v_2 m}(\rd)\to \mathcal{M}^{p,q}_m(\rd)$, $\mathcal{M}^{p,q}_m(\rd)\to \mathcal{M}^{p,q}_{m/v_2}(\rd)$ and the map $t\mapsto a^w(t,\cdot)$ is strongly continuous  on the these spaces.\par
%Similarly, the operators $(\partial_{x_j} a)^w(t,x,D)$ and $(\partial_{\xi_j} a)^w(t,x,D)$ are bounded $\mathcal{M}^p_{v_1\cdot m}\to \mathcal{M}^p_m$, $\mathcal{M}^p_m\to \mathcal{M}^p_{v_{-1}\cdot m}$ and as functions of $t$ are strongly continuous.
\end{proposition}
\begin{proof}
Since the operators $y^\alpha D^\beta$, $|\alpha|+|\beta|\leq 2$, are bounded 
%$ \mathcal{M}^p_{v_2\cdot m}\to \mathcal{M}^p_m$,
 $\mathcal{M}^{p,q}_m(\rd)\to \mathcal{M}^{p,q}_{m/v_2}(\rd)$ (Proposition \ref{proiniziale}), by a $(j-1)$-th order Taylor expansion of $a_j(t,\cdot)$ at $(0,0)$, $j=1,2$, we are reduced to prove the result for the remainder
\begin{multline}\label{und1}
b(t,y,\eta)= 2\sum_{|\alpha|+|\beta|= 2} \frac{1}{\alpha!\beta!}\int_0^1 (1-\tau) \big(\partial^\alpha_x\partial^\beta_\xi a_2\big)(t,\tau(y,\eta))\,d\tau\, y^\alpha\eta^\beta\\
+\sum_{|\alpha|+|\beta|=1}\int_0^1 (\partial^{\alpha}_x\partial^\beta_\xi a_1)(t,\tau(y,\eta))\,d\tau\, y^\alpha\eta^\beta+a_0(t,y,\eta).
\end{multline}
Using repeatedly \eqref{comm}, \eqref{comm1} and \eqref{comm2} we can write 
\begin{align}\label{und2}
b^w(t,y,D_y)=& \sum_{|\alpha|+|\beta|\leq 2} a^w_{\alpha,\beta}(t,y,D_y) y^\alpha D_y^\beta\\
=&\sum_{|\alpha|+|\beta|\leq 2} y^\alpha D_y^\beta \tilde{a}^w_{\alpha,\beta}(t,y,D_y), \nonumber
\end{align}
where $a_{\alpha,\beta}(t,\cdot)$ and $\tilde{a}_{\alpha,\beta}(t,\cdot)$ are continuous maps in $M^{\infty,1}_{1\otimes v}(\rdd)$ for the narrow convergence by Proposition \ref{prod} and the assumption \eqref{due}. Hence the corresponding operators are strongly bounded on $\cM^{p,q}_m(\rd)$ by Proposition \ref{pro3.2}. \par
This concludes the proof.
\end{proof}\par

\end{lemma}
\section{A class of generalized localization operators}
We consider the Hamiltonian flow $(x^t,\xi^t)$, as a function of $t\in[0,T]$, $x,\xi\in\rd$, given by the solution of
\begin{equation}\label{sistema}
\begin{cases}
\dot{x}^t=\nabla_\xi a_2(t,x^t,\xi^t)\\
\dot{\xi}^t=-\nabla_x a_2(t,x^t,\xi^t) \\
x^0(x,\xi)=x,\ \xi^0(x,\xi)=\xi.
\end{cases}
\end{equation}
Under the assumptions of Theorems \ref{teo1}-\ref{teo4} we see that the functions $\partial^\alpha_{x,\xi} a_2(t,x,\xi)$, $|\alpha|=2$, are continuous  by Proposition \ref{proregolarita}. Moreover they are bounded, because $M^{\infty,1}_{1\otimes v}(\rdd)\subset M^{\infty,1}(\rdd)\subset L^\infty(\rdd)$. Hence the solution of the above initial value problem exists globally in time for every initial datum, and the flow is a map of class $C^1$ in all variables $t,x,\xi$. The map $\chi(t,s):(x^s,\xi^s)\to(x^t,\xi^t)$ is moreover symplectic.\par
Further consider the real-valued phase $\psi(t,x,\xi)$ defined by
\begin{equation}\label{fase}
\psi(t,x,\xi)= \int_0^t \Big(\xi^s\cdot(a_2)_\xi (s,x^s,\xi^s)-a_2(s,x^s,\xi^s)-a_1(s,x^s,\xi^s)\Big)\,ds.
\end{equation}
We now introduce the class of operators used in the next Section for the construction of the parametrix (cf.\ \cite{daubechies,tataru}).
\begin{proposition}\label{lemma2}
Let $g\in \cS(\rd)$. Let $[0,T]\times\rdd\ni (t,x,\xi)\mapsto G(t,x,\xi,\cdot)$ be a continuous map in $\cS'(\rd)$ (weakly), such that $|V_gG(t,x,\xi,\cdot)|\leq H$ for every $t,x,\xi$, for some function $H\in L^1_v(\rdd)$ (in particular we have $G(t,x,\xi,\cdot)\in M^1_v(\rd)$). \par For every $s,t\in[0,T]$, the operator 
\[
T_{t,s}f= \iint e^{i\psi(t,x,\xi)-i\psi(s,x,\xi)} \pi(x^t,\xi^t)G(t,x,\xi,\cdot) V_g f(x^s,\xi^s)\, dx\,d\xi
\]
is strongly continuous on $\mathcal{M}^p_{m}(\rd)$, $1\leq p\leq\infty$, $m\in\mathcal{M}_v$.
\end{proposition}
\begin{proof}
We first apply the change of variable $(y,\eta)=\chi(t,0)(x,\xi)$, which has Jacobian $=1$, and we call again $(x,\xi)$ for the new variables; we obtains
\begin{multline*}
T_{t,s}f= \iint e^{i\psi(t,\chi(t,0)^{-1}(x,\xi))-i\psi(s,\chi(t,0)^{-1}(x,\xi))}\\
\times \pi(x,\xi)G(t,\chi(t,0)^{-1}(x,\xi),\cdot) V_g f(\chi(t,s)^{-1}(x,\xi))\, dx\,d\xi.
\end{multline*}
We can write 
\[
V_g\circ T_{t,s}=   U_{t,s} R_{t,s}V_g
\]
where $V_g: \mathcal{M}^p_m(\rd)\to L^p_m(\rdd)$ is of course bounded, and the operators $U_{t,s}$, $R_{t,s}$ are defined as follows. The operator
\[
R_{t,s} F= e^{i\psi(t,\chi^{-1}(t,0)\cdot)-i\psi(s,\chi^{-1}(t,0)\cdot)}F\circ \chi^{-1}(t,s),\quad F\in L^p_m(\rdd),
\]
is strongly continuous on $L^p_m(\rdd)$, $1\leq p<\infty$, and on the closure of the Schwartz space in $L^{\infty}_m(\rdd)$; this is straightforward to check. \par
The operator
\[
U_{t,s} F(z)= \iint \langle \pi(x,\xi) G(t,\chi(t,0)^{-1}(x,\xi),\cdot), \pi(z)g\rangle F(x,\xi)\, dx\,d\xi,\quad F\in L^p_m(\rdd),
\] 
is bounded on $L^p_m(\rdd)$, because by the assumption on $G$ and \eqref{covarianza} we can dominate its integral kernel by a convolution kernel in $L^1_v(\rdd)$:
\[
|\langle \pi(w) G(t,\chi(t,0)^{-1}(x,\xi),\cdot), \pi(z)g\rangle|\leq H(w-z).
\]
To prove the strong continuity of $U_{t,s}$ we can assume $F\in L^1_v$, and the conclusion then follows by the dominated convergence theorem, as in the proof of Proposition \ref{pro3.2}. 
\end{proof}

\section{Proofs of the main results (Theorems \ref{teo1}--\ref{teo4})}
We first construct a parametrix for the Cauchy problem \eqref{equazione}, in the form of generalized localization operator as in the previous section.  
\subsection{Construction of a parametrix}
\begin{theorem}\label{teopara} Assume the hypotheses of Theorem \ref{teo1}. For $0\leq s\leq t\leq T$ there are operators $\tilde{S}(t,s)$, ${K}(t,s)$ satisfying the following properties: 
\begin{itemize}
\item[(1)] $\tilde{S}(t,s)$ and ${K}(t,s)$ are strongly continuous on $\mathcal{M}^p_{m}(\rd)$, $1\leq p\leq\infty$;
\item[(2)] $\tilde{S}(s,s)={\rm Id}$ for every $0\leq s\leq T$;
\item[(3)] $\big(D_t+a^w(t,x,D_x)\big)\tilde{S}(t,s) f=K(t,s)f$ for every $f\in \mathcal{M}^p_{m}(\rd)$, $0\leq s\leq t\leq T$.
\end{itemize}
\end{theorem}

\begin{proof}
We set 
\[
\tilde{S}(t,s)f=\iint e^{i\psi(t,x,\xi)-i\psi(s,x,\xi)} \pi(x^t,\xi^t)g(\cdot) V_g f(x^s,\xi^s)\, dx\,d\xi,
\]
where $g$ is any Schwartz function, with $\|g\|_{L^2}=1$. Property (2) then follows from the inversion formula \cite[Corollary 3.2.3]{book}
\[
f= \iint \pi(x,\xi)g(\cdot) V_g f(x,\xi)\, dx\,d\xi
\]
and the change of variable $(y,\eta)= (x^s,\xi^s)$, which has Jacobian $=1$.\par
The strong continuity of $\tilde{S}(t,s)$ in (1) is a consequence of Proposition \ref{lemma2}. Let us compute $
\big(D_t+a^w(t,x,D_x)\big)\tilde{S}(t,s) f$. We have by an explicit computation, using \eqref{sistema} and \eqref{fase}, 
\begin{align}\label{p4}
&D_t \Big(e^{i\psi(t,x,\xi)}\pi(x^t,\xi^t) g\Big)(y)=e^{i\psi(t,x,\xi)}\pi(x^t,\xi^t)[\partial_t\psi g+\dot{\xi}^t\cdot yg-\xi^t\cdot \dot{x}^tg-\dot{x}^t\cdot (-i\nabla_y) g]  \\
&=e^{i\psi(t,x,\xi)}\pi(x^t,\xi^t)[-a_2(t,x^t,\xi^t)g-\nabla_x a_2(t,x^t,\xi^t)\cdot y g\nonumber\\
&\qquad\qquad\qquad\qquad\qquad\qquad\qquad-\nabla_{\xi} a_2(t,x^t,\xi^t)\cdot (-i\nabla_y) g-a_1(t,x^t,\xi^t)g]\nonumber
\end{align}
whereas (cf.\ the symplectic invariance of the Weyl calculus in \cite[Theorem 18.5.9]{hormanderIII})
\[
a^w(t,y,D_y) \Big(e^{i\psi(t,x,\xi)}\pi(x^t,\xi^t) g\Big)=e^{i\psi(t,x,\xi)}\pi(x^t,\xi^t)a^w(t,y+x^t,D_y+\xi^t)g.
\]
Hence we obtain
\[
\Big(D_t +a^w(t,y,D_y) \Big)e^{i\psi(t,x,\xi)}\pi(x^t,\xi^t)g=e^{i\psi(t,x,\xi)}\pi(x^t,\xi^t)[b^w(t,y,D_y) g],
\]
with
\begin{align}\label{p5}
b(t,x,\xi,y,\eta)&= a(t,y+x^t,\eta+\xi^t)-a_2(t,x^t,\xi^t)-\nabla_x a_2(t,x^t,\xi^t)y\\
&\qquad\qquad\qquad\qquad\qquad-\nabla_\xi a_2(t,x^t,\xi^t)\eta-a_1(t,x^t,\xi^t)\nonumber\\
&=2\sum_{|\alpha|+|\beta|= 2}\frac{1}{\alpha!\beta!}\int_0^1(1-\tau)\partial^\alpha_x\partial^\beta_\xi a_2(t,(x^t,\xi^t)+\tau (y,\eta))\, d\tau\, y^\alpha \eta^\beta\nonumber\\
&\quad+\sum_{|\alpha|+|\beta|=1}\int_0^1\partial^\alpha_x\partial^\beta_\xi a_1(t,(x^t,\xi^t)+\tau (y,\eta))\, d\tau\, y^\alpha \eta^\beta\nonumber\\
&\quad+a_0(t,x^t+y,\xi^t+\eta).\nonumber
\end{align}
Using repeatedly \eqref{comm1} and \eqref{comm2}, the assumption \eqref{due} and Proposition \ref{prod} we can write 
\begin{equation}\label{eq0}
b^w(t,x,\xi,y,D_y)= \sum_{|\alpha|+|\beta|\leq 2} a^w_{\alpha,\beta}(t,x,\xi,y,D_y) y^\alpha D_y^\beta
\end{equation}
where $a_{\alpha,\beta}(t,x,\xi,y,\eta)$ are a family of symbols in $M^{\infty,1}_{1\otimes v}(\rdd)$, continuous with respect to the narrow convergence as a function of $t,x,\xi$. \par
Hence we get
\[
\big(D_t+a^w(t,x,D_x)\big)\tilde{S}(t,s) f=K(t,s)f
\]
with 
\[
K(t,s)f=\iint e^{i\psi(t,x,\xi)-i\psi(s,x,\xi)} \pi(x^t,\xi^t)[b^w(t,x,\xi,y,D_y)g] V_g f(x^s,\xi^s)\, dx\,d\xi.
\] 
To see that this operator enjoys the properties in (1), by Proposition \ref{lemma2} we have to verify that $G(t,x,\xi,\cdot):=b^w(t,x,\xi,y,D_y)g$ is continuous as a function of $t,x,\xi$ in $\cS'(\rd)$, which is clear, and that 
\begin{equation}\label{p6}
\sup_{z\in\rdd}|V_g G(t,x,\xi;z,w)|\leq H(w)\quad t\in[0,T],\, (x,\xi),w\in\rdd
\end{equation}
for some function $H\in L^1_v(\rdd)$ independent of $t,x,\xi$. To this end observe that, using \eqref{eq0}, we are left to prove that if $c(t,x,\xi,\cdot,\cdot)$ is a family of symbols in $M^{\infty,1}_{1\otimes v}(\rdd)$, continuous for the narrow convergence with respect to $t,x,\xi$, it turns out, for $g,\gamma\in\cS(\rd)$, 
\[
|\langle c^w(t,x,\xi,y,D_y)\gamma,\pi(w)g\rangle|\leq H(w)
\]
for some function $H\in L^1_v(\rdd)$. Lemma \ref{changewind} reduces matters to the case $g=\gamma$, where one can then conclude by applying Lemma \ref{lemmag}.
\end{proof}

\subsection{Proof of Theorem \ref{teo1}}
\subsubsection{Existence} Let us prove that there exists a strongly continuous  operator $S(t,s)$ on $\mathcal{M}^p_m(\rd)$ such that $S(s,s)={\rm Id}$, and
\begin{equation}\label{p1}
(D_t+a^w(t,x,D_x))S(t,s)u_0=0, \quad u_0\in \mathcal{M}^p_m(\rd).
\end{equation}
First we observe that for $K(t,s)$ as in Theorem \ref{teopara}, for every $u_0\in \mathcal{M}^p_m(\rd)$ there exists a unique solution $v\in C([0,T],\mathcal{M}^p_m(\rd))$ to the Volterra equation
\[
v(t)=-K(t,0) u_0-i\int_0^t K(t,s) v(s)\, ds.
\]
This is a consequence of the contraction mapping theorem applied in the space $C([0,T],\cM^p_m(\rd))$ endowed with the norm $\sup_{t\in[0,T]}e^{-\lambda t}\|u(t)\|_{\cM^p_m}$, with 
\[
\lambda>\sup_{0\leq s\leq t\leq T}\|K(t,s)\|_{\cM^p_m \to \cM^p_m}.
\] Moreover one obtains $\|v(t)\|_{\mathcal{M}^p_m}\lesssim \|u_0\|_{\mathcal{M}^p_m}$, $t\in [0,T]$.\par
 If $v$ is such a solution, we then set
\[
S(t,s)u_0=\tilde{S}(t,0)u_0+i\int_0^t \tilde{S}(t,s) v(s)\, ds,
\]
where $\tilde{S}(t,s)$ is the parametrix in Theorem \ref{teopara}. It is then straightforward to check that \eqref{p1} holds true, as well as the claimed boundedness of the propagator. 

\subsubsection{Uniqueness} 
We have to prove that if $u\in C([0,T],\mathcal{M}^p_m(\rd))$ is a solution to the equation in \eqref{equazione} and $u_0=u(0)=0$, then $u=0$. Let $g\in\cS(\rd)$. By Proposition \ref{pro3.4} we have in particular $u\in C^1([0,T], \mathcal{M}^p_{ m/v_2}(\rd))\subset C^1([0,T], \cS'(\rd))$, so that $V_g u(t)(x^t,\xi^t)$ is continuously differentiable as a function of $t,x,\xi$. By the fundamental theorem of calculus we have the pointwise equality 
\[
e^{-i\psi(t,x,\xi)} V_g(u(t))(x^t,\xi^t)=i\int_0^t D_s \Big(e^{-i\psi(s,x,\xi)} V_g(u(s))(x^s,\xi^s)\Big)\, ds.
\]
On the other hand it turns out
\begin{align*}
D_s \Big(e^{-i\psi(s,x,\xi)}& V_g(u(s))(x^s,\xi^s)\Big)=D_s \langle u(s),e^{i\psi(t,x,\xi)}\pi(x^s,\xi^s)g\rangle\\
&=\langle D_s u(s),e^{i\psi(s,x,\xi)}\pi(x^s,\xi^s)g\rangle-\langle u(s), D_s \big(e^{i\psi(s,x,\xi)}\pi(x^s,\xi^s)g \big)\rangle 
\end{align*}
where we used that the map $s\mapsto e^{i\psi(s,x,\xi)}\pi(x^s,\xi^s)g$ is continuously differentiable in $\cS(\rd)$. Using \eqref{p4} and \eqref{p5} we obtain\footnote{The duality pairing is understood for the pairs of spaces $\cM^p_m-\cM^1_v$ and $\cM^p_{m/v_2}-\cM^1_{v_2v}$; cf. Proposition \ref{pro3.4}.}
\begin{align*}
\langle u(s)&, D_s \Big(e^{i\psi(s,x,\xi)}\pi(x^s,\xi^s)g \Big)\rangle\\
&=\langle u(s),-a^w(s,x,D)e^{i\psi(s,x,\xi)}\pi(x^s,\xi^s)g+ e^{i\psi(s,x,\xi)}\pi(x^s,\xi^s)G(s,x,\xi,\cdot)\rangle\\
&=\langle -a^w(s,x,D)u(s),e^{i\psi(s,x,\xi)}\pi(x^s,\xi^s)g\rangle+\langle u(s),e^{i\psi(s,x,\xi)}\pi(x^s,\xi^s)G(s,x,\xi,\cdot)\rangle,
\end{align*}
where $G(s,x,\xi,\cdot)= b^w(s,x,\xi,y,D_y)g$ (same notation as in the proof of Theorem \ref{teopara}). \par
Summing up we get
\begin{align}\label{p8}
e^{-i\psi(t,x,\xi)} V_g(u(t))(x^t,\xi^t)= i\int_0^t \langle \underbrace{D_s u(s)+a^w(s,x,D)u(s)}_{=0}, e^{i\psi(s,x,\xi)}\pi(x^s,\xi^s)g\rangle\\
-  \langle u(s),e^{i\psi(s,x,\xi)}\pi(x^s,\xi^s)G(s,x,\xi,\cdot)\rangle\, ds.  \nonumber
\end{align}
Using \eqref{changewind} we can estimate the last term as 
\[
|\langle u(s),e^{i\psi(s,x,\xi)}\pi(x^s,\xi^s)G(s,x,\xi,\cdot)\rangle|\lesssim \Big(|V_g u(s)|\ast |V_g G(s,x,\xi,\cdot)|\Big)(x^s,\xi^s).
\]
Since \eqref{p6} holds true for some $H\in L^1_v(\rdd)$, by Young's inequality we obtain
\[
\|\langle u(s),e^{i\psi(s,x,\xi)}\pi(x^s,\xi^s)G(s,x,\xi,\cdot)\rangle\|_{L^p_{m}(\R^{2d}_{x,\xi})}\lesssim \|V_g u(s)\|_{L^p_m}=\|u(s)\|_{\mathcal{M}^p_m}.
\]
We can then estimate in \eqref{p8}
\[  
\|u(t)\|_{\mathcal{M}^p_m}=\|V_g(u(t))\|_{L^p_m}\leq C\int_0^t \|u(s)\|_{\mathcal{M}^p_m}\, ds,\quad t\in[0,T],
\]
and therefore $u(t)=0$ for every $t\in[0,T]$, by Gronwall's inequality.
\subsection{Proof of Theorem \ref{teo2}}
A carefully inspection of the proof of Theorem \ref{teo1} and the needed preliminary results shows that the only point were we used the condition on $\partial^\alpha a_2(t,\cdot)$ for $|\alpha|=3,4$, and $\partial^\alpha a_1(t,\cdot)$ for $|\alpha|=2$, is to pass to the algebraic expression in \eqref{und1} to the corresponding quantization in the form \eqref{und2} -- where, using \eqref{comm}--\eqref{comm2}, additional derivatives fall on the symbol-- and similarly in the proof of Theorem \ref{teopara} (from \eqref{p5} to \eqref{eq0}). However, when the symbol $a(t,x,\xi)$ has the special form in Theorem \ref{teo2} the factorization at the level of symbol for $\sigma_j(t,\xi)$, $V_j(t,x)$ gives a corresponding exact factorization at the level of operators, and therefore the conditions $\partial^\alpha a_2(t,\cdot)\in M^{\infty,1}_{1\otimes v}(\rd)$, $|\alpha|=2$, and $\partial^\alpha a_1(t,\cdot)\in M^{\infty,1}_{1\otimes v}(\rd)$, $|\alpha|=1$, suffice in that case.
\subsection{Proof of Theorem \ref{teo3}}
 It suffices to prove that, under the present assumptions, the results in Proposition \ref{lemma2} and therefore those in Theorem \ref{teopara} still hold for the modulation spaces $\cM^{p,q}_m(\rd)$, $1\leq p,q\leq\infty$. In fact, one is left to prove that the map $F\mapsto F\circ\chi^{-1}(t,s)$ is bounded of $L^{p,q}_m(\rdd)$. Now, we have $a_2(t,x,\xi)=\sigma_2(t,\xi)$, so that the symplectic map $\chi$ (we omit the dependence on $s,t$ for brevity) has the special form $\chi(y,\eta)=(\beta(y,\eta),\eta)$, $\chi^{-1}(y,\eta)=(\tilde{\beta}(y,\eta),\eta)$,  with ${\rm det}\, \partial\beta/\partial y=1$. Hence the change of variable $y=\beta(y',\eta)$ (for fixed $\eta$) gives
\[
\|F(\tilde{\beta}(y,\eta),\eta)m(y,\eta)\|_{L^q(\rd_\eta;L^p(\rd_y))}=\|F(y',\eta)m(\beta(y',\eta),\eta)\|_{L^q(\rd_\eta;L^p(\rd_{y'}))}.
\]
However $m(\beta(y',\eta),\eta)=m(\chi(y',\eta))\asymp m(y',\eta)$ since $\chi$ is a Lipschitz map, and \eqref{condpesi3} holds by assumption. This concludes the proof.
\subsection{Proof of Theorem \ref{teo4}}
We need the following result from \cite[Proposition 2.5]{cnr3}. 
\begin{proposition}\label{sopra}
Let $h\in C^\infty(\rd\setminus\{0\})$ be positively homogeneous of degree $\kappa>0$, i.e.\ $h(\lambda x)=\lambda^\kappa h(x)$ for $x\not=0$, $\lambda>0$, and $\chi\in C^\infty_0(\rd)$. Set $f=h\chi$. Then, for $g\in\cS(\rd)$ there exists a constant $C>0$ such that
\[
|V_g f(x,\xi)|\leq C(1+|\xi|)^{-\kappa-d},\quad x,\xi\in\rd.
\]
 \end{proposition}
 One can choose a cut-off function $\chi\in C_0^\infty(\rd)$, $\chi=1$ in a neighborhood of the origin,  and write the equation as 
\[
D_tu+a^w(t,x,D_x) u=0
\]
with $a(t,x,\xi)=\sum_{j=0}^2 a_j(t,x,\xi)$, where $a_2(t,x,\xi)=|\xi|^\kappa(1-\chi(\xi))+V_2(t,x)$, $a_1(t,x,\xi)=V_1(t,x)$, $a_0(t,x,\xi)=\chi(\xi)|\xi|^\kappa+V_0(t,x)$. By Proposition \ref{sopra} we have $a_0(t,\cdot)\in M^{\infty,1}_{1\otimes v_r}$ for every $0\leq r<\kappa$. Hence, the conclusion follows at once from Theorems \ref{teo2}, \ref{teo3} (with $v=v_r$). 
\begin{remark}\label{remark5.2}\rm
Let us observe that the initial value problem for $D_t+\Delta$ is not wellposed in $M^\infty$. In fact, if the initial datum is $u_0=\delta\in M^\infty$, the solution $S(t,0)u_0=\Fur^{-1}(e^{it|\cdot|^2})$ verifies (by the invariance of $M^\infty(\rd)$ under Fourier transform)
\[
\|S(t,0)u_0-u_0\|_{M^\infty}\gtrsim \|\Fur(S(t,0)u_0-u_0)\|_{M^\infty}=\|e^{it|\cdot|^2}-1\|_{M^\infty}.
\]
Now, using the window $g(x)=e^{-\pi|\xi|^2}$, an explicit computation (cf.\ \cite[Theorem 6]{bertinoro3}, modified according to our normalization) shows that
\[
|V_g(e^{it|\cdot|^2})(x,\xi)|=(1+t^2)^{-d/4}\exp\Big(-\frac{|\xi-tx|^2}{4\pi(1+t^2)}\Big),
\]
which, if $t\not=0$, tends to $0$ as $|x|\to+\infty$ for every fixed $\xi\in\rd$. On the other hand, 
\[
|V_g(1)(x,\xi)|=e^{-|\xi|^2/(4\pi)}
\]
is independent of $x$. Hence we see that $\sup_{x\in\rd}|V_g(e^{it|\cdot|^2} -1)(x,\xi)|\geq e^{-|\xi|^2/(4\pi)}$ for every  $\xi\in\rd$. As a consequence, $\|S(t,0)u_0-u_0\|_{M^\infty}\geq 1$ does not tend to zero as $t\to0$.

\end{remark}
\section{Nonlinear Schr\"odinger equations} 
In this section we briefly discuss the extension of the above results in presence of a non-linearity of the type $F(u)$, where the function $F:\mathbb{C}\to\mathbb{C}$ is entire real-analytic, with $F(0)=0$ ($F(z)$ has a Taylor expansion in $z,\overline{z}$, valid in the whole complex plane). In particular we can take a polynomial in $z,\overline{z}$.\par To avoid repetitions we summarize the results in a unique statement.
\begin{theorem} Each of the Theorems \ref{teo1}--\ref{teo4} admits a nonlinear variant. Namely, under those assumptions, we have local wellposedness in $\cM^1_m$ or even in $M^{p,1}_m$ (according to the corresponding linear result), for $1\leq p\leq\infty$, $m\in\cM_v$, $m\gtrsim 1$, of the initial value problem for the nonlinear equation $Lu=F(u)$, where $L=D_t+ a^w(t,x,D_x)$ is the linear operator, with $F(u)$ as specified above. 
\end{theorem} 
It is understood that $v=v_r$, $0\leq r<\kappa$, in the case of Theorem  \ref{teo4}. 
\begin{proof}
The argument is standard; cf.\ \cite[Theorem 6.1]{baoxiang}. Namely, we apply the contraction mapping theorem in the space $X:=C([0,T_0],\cM^{1}_m(\rd))$, with $T_0$ small enough (or $X:=C([0,T_0],\cM^{p,1}_m(\rd))$, when appropriate). \par By the Duhamel principle we can rewrite the semilinear equation in integral form 
\[
u(t)=S(t,0) u_0- i\int_0^t S(t,s) F(u(s))\,d s
\]
where $S(t,s)$, $0\leq s\leq  t\leq T$, is the linear propagator corresponding to initial data at time $s$.  
The classical iteration scheme works in $X$ if the following properties are verified:
\begin{itemize}
\item[\bf a)] $S(t,s)$ is strongly continuous on $\cM^{1}_s(\rd)$ for $0\leq s\leq t\leq T$ (which also implies a uniform bound for the operator norm with respect to $s,t$, by the uniform boundedness principle);
\item[\bf b)] We have $\|F(u)-F(v)\|_X\leq C\|u-v\|_X$, for $u,v\in X$ in every fixed ball.
\end{itemize}
The estimate in {\bf b)} was proved in \cite[Formula (28)]{bertinoro12} in the case $m(x,\xi)=\langle \xi\rangle^s$, $s\geq0$, but the same proof extends to any weight $m\in \cM_v$ satisfying $m\gtrsim 1$, since one just uses the fact that $\cM^1_m$ is a Banach algebra for pointwise multiplication (the same holds for $\cM^{p,1}_m$). \par
\par
As far as {\bf a)} is concerned, it follows from the above linear results that $S(t,s)$ is bounded on $\cM^{1}_m$ uniformly with respect to $s,t$ and it is strongly continuous in $\cM^{1}_m$ as a function of $t$, for fixed $s$. The same holds for $s,t$ exchanged, because the equation is time-reversible. To prove its strong continuity jointly in $(t,s)$ we observe that if $s'\leq s\leq t$, $u_0\in \cM^1_m(\rd)$,
 \begin{align*}
\|S(t,s)u_0-S(t,s')u_0\|_{\cM^1_m}&\leq C_1\|u_0-S(s,s')u_0\|_{\cM^1_m}.
\end{align*} 
Hence, the map $s\mapsto S(t,s)u_0$ is in fact continuous in $\cM^1_m(\rd)$ uniformly with respect to $t$. This yields the strong continuity of $S(t,s)$ on $\cM^1_m$, as a function of $s,t$ and concludes the proof.
\end{proof}

%For acknowledgements section, please don't number the section, please begin it with \section*{Acknowledgements}
%\section*{Acknowledgments} We would like to thank you for \textbf{following the instructions above} very closely in advance. It will definitely save us lot of time and expedite the process of your paper's publication.

% You may incorporate your references as follows in your main tex file.
% Using BibTex is not recommended but can be handled.

\medskip
% The data information below will be filled by AIMS editorial staff
Received xxxx 20xx; revised xxxx 20xx.
\medskip

\end{document}